\documentclass{birkjour}
 \newtheorem{theorem}{Theorem}[section]
 \newtheorem{corollary}[theorem]{Corollary}

 \theoremstyle{definition}
 \newtheorem{definition}[theorem]{Definition}
 \theoremstyle{remark}
 \newtheorem{remark}[theorem]{Remark}
 
 \numberwithin{equation}{section}

\begin{document}

%
%
%
%
%
%
%
%
%

\title[Sequentially Congruent Partitions]
 {Sequentially Congruent Partitions\\  and Related Bijections}

\author{Maxwell Schneider}

\address{
Honors Program\newline
University of Georgia\newline
Athens, Georgia 30602}
\email{maxwell.schneider@uga.edu}%

\author{Robert Schneider}
\address{Department of Mathematics\newline
University of Georgia\newline
Athens, Georgia 30602}
\email{robert.schneider@uga.edu}


\dedicatory{In honor of George E. Andrews on his 80th birthday}

\begin{abstract}
We study a curious class of partitions, the parts of which obey an exceedingly strict congruence condition we refer to as ``sequential congruence'': the $m$th part is congruent to the $(m+1)$th part modulo $m$, with the smallest part congruent to zero modulo the length of the partition. It turns out these obscure-seeming objects are embedded in a natural way in partition theory. We show that sequentially congruent partitions with largest part $n$ are in bijection with the partitions of $n$. Moreover, we show sequentially congruent partitions induce a bijection between partitions of $n$ and partitions of length $n$ 
whose parts obey a strict ``frequency congruence'' condition --- the frequency (or multiplicity) of each part is divisible by that part --- and prove families of similar bijections, connecting with G. E. Andrews's theory of partition ideals.
\end{abstract}

\maketitle
\section{Introduction}


Here we consider a somewhat exotic subset of {integer partitions}, which turns out to be naturally embedded in partition theory. 

Let $\mathcal P$ denote the set of {\it partitions}, with elements $\lambda = (\lambda_1, \lambda_2, ..., \lambda_r)$, $\lambda_1 \geq \lambda_2 \geq ...\geq \lambda_r\geq 1$, including the {\it empty partition} $\emptyset$. Alternatively, following Andrews \cite{Andrews_theory}, Fine \cite{Fine} and other authors, one sometimes writes 
$\lambda = (1^{m_1}\  2^{m_2}\  3^{m_3}\  ...)$ with $m_i=m_i(\lambda)$ the {\it frequency} (or {\it multiplicity}) of $i\in \mathbb N$ as a part of $\lambda$, setting $m_i(\emptyset)=0$ for all $i$. Furthermore, for a given partition $\lambda$, let $|\lambda|$ denote its {\it size} (sum of the parts) and $\ell(\lambda):=r$ denote its {\it length} (number of parts), with the conventions $|\emptyset|:=0,\  \ell(\emptyset):=0$. 

We define the set $\mathcal S \subset \mathcal P$ of {\it sequentially congruent partitions} as follows.

\begin{definition}
We define a partition $\lambda$ to be {\it sequentially congruent} if 
the following congruences between the parts are all satisfied:
$$\lambda_1 \equiv \lambda_2\  (\operatorname{mod}\  1),\  
\lambda_2 \equiv \lambda_3\  (\operatorname{mod}\  2),\  \lambda_3 \equiv \lambda_4\  (\operatorname{mod}\  3),\  ...\  ,$$ 
$$\lambda_{r-1} \equiv \lambda_r\  (\operatorname{mod}\  r-1),$$
and for the smallest part, $\lambda_r \equiv 0\  (\operatorname{mod}\  r).$ 
\end{definition}

For example, the partition $(20,17,15,9,5)$ 
is sequentially congruent, because $20 \equiv 17\  (\operatorname{mod}\  1)$ trivially, $17\equiv 15\  (\operatorname{mod}\  2)$, $15\equiv 9\  (\operatorname{mod}\  3)$, $9\equiv 5\  (\operatorname{mod}\ 4)$, and finally $5\equiv 0\  (\operatorname{mod}\  5)$. On the other hand, $(21,18,16,10,6)$ is {\it not} sequentially congruent, for while the first four congruences still hold, clearly $6\not\equiv 0\  (\operatorname{mod}\  5)$. Note that increasing the largest part $\lambda_1$ of any $\lambda \in \mathcal S$ yields another partition in $\mathcal S$, as does adding or subtracting a fixed integer multiple of the length $r$ to all its parts, so long as the resulting parts are still positive.

No doubt, this strict congruence restriction on the parts hardly appears natural. However, it turns out sequentially congruent partitions are in one-to-one correspondence with the entire set $\mathcal P$.

\section{Bijections Between $\mathcal S$ and $\mathcal P$}\label{Sect2}
Let $\mathcal P_n$ denote the set of partitions of $n$, as usual let $p(n)=\#\mathcal P_n$ (with $\#Q$ the cardinality of a set $Q$), and let $\mathcal S_{\operatorname{lg}=n}$ denote sequentially congruent partitions $\lambda'$ whose largest part $\lambda_{1}'$ equals $n$.

\begin{theorem}\label{Cthm1}
There exists a 
bijection $\pi$ 
between the set $\mathcal P$ and the set $\mathcal S$ such that 
$$\pi(\mathcal P_n)=\mathcal S_{\operatorname{lg}=n}.$$
Moreover, we have 
$$\#\mathcal S_{\operatorname{lg}=n}=p(n).$$
\end{theorem}


\begin{proof}
We prove the theorem 
directly by construction. 

For partition $\lambda=(\lambda_1, \lambda_2,...,\lambda_i,...,\lambda_r)$, one constructs a sequentially congruent dual $$\lambda'=(\lambda_1',\lambda_2',...,\lambda_i',...,\lambda_r')$$ 
by taking the parts equal to
\begin{equation}\label{Cconstruction1}
\lambda_i'=i\lambda_i+\sum_{j=i+1}^{r}\lambda_j.
\end{equation}
Note that $\lambda_r'\equiv 0\  (\operatorname{mod}\  r)$ as $\sum_{j=r+1}^{r}$ is empty; the other congruences between successive parts of $\lambda'$ are also immediate from equation (\ref{Cconstruction1}). 

Let us take $$\pi \colon \mathcal P \to \mathcal S$$ to be the map defined by this construction, with $\lambda'=\pi(\lambda)$. 
The above argument establishes, in fact, that we have more strongly $\pi\colon \mathcal P_n \to \mathcal S_{\text{lg}=n}$. 

Conversely, given a sequentially congruent partition $\lambda'$, one can recover the dual partition $\lambda$ by working from right-to-left. Begin by computing the smallest part 
\begin{equation}\label{Cconstruction1.5}
\lambda_r=\frac{\lambda_r'}{r},\end{equation}
then compute $\lambda_{r-1},\lambda_{r-2},...,\lambda_1$ in this order by taking
\begin{equation}\label{Cconstruction2}
\lambda_i=\frac{1}{i}\left( \lambda_i'-\sum_{j=i+1}^{r}\lambda_j \right).\end{equation}
We define the inverse map $\pi^{-1}$ from the algorithm in \eqref{Cconstruction1.5} and \eqref{Cconstruction2}, i.e., $\pi^{-1}(\lambda')=\lambda$: 
\begin{equation}\label{piinv} \pi^{-1}\colon \mathcal S \to \mathcal P.\end{equation} Noting that the uniqueness of $\lambda$ implies the uniqueness of $\lambda'$, and vice versa, the bijection between $\mathcal S$ and $\mathcal P$ follows from this two-way construction. 

Furthermore, since $\lambda_1'=|\lambda|$, then every partition $\lambda$ of $n$ corresponds to a sequentially congruent partition $\lambda'$ with largest part $n$, and vice versa. 
\end{proof}

The sets $\mathcal P$ and $\mathcal S$ enjoy another interrelation that can be used to compute the coefficients of infinite products. Now, it is a rewriting of Equation 22.16 in Fine \cite{Fine} that for a function $f\colon \mathbb N \to \mathbb C$ and $q\in \mathbb C$ with $f,q$ chosen such that the product converges absolutely, we have
\begin{equation}\label{Czeta_thm}
\prod_{n=1}^{\infty}(1-f(n)q^n)^{-1}=\sum_{\lambda \in \mathcal P}q^{|\lambda|}\prod_{i\geq 1}f(i)^{m_i},
\end{equation}
where $m_i=m_i(\lambda)$ is the frequency of $i$ as a part of $\lambda$, and the sum on the right is taken over all partitions $\lambda$. Of course the canonical case would be, for $|q|<1$, the identity 
\begin{equation}\label{Czeta_thm3}
\prod_{n=1}^{\infty}(1-xq^n)^{-1}=\sum_{\lambda \in \mathcal P}x^{\ell(\lambda)}q^{|\lambda|},
\end{equation}
which enjoys many beautiful $q$-series representations (see \cite{Andrews_theory, Berndt, Fine})

It follows from an extension of \eqref{Czeta_thm} in \cite{Schneider_zeta} that the product on the left side of \eqref{Czeta_thm} can also be expressed as a sum over {\it sequentially congruent} partitions. 

Let $\operatorname{lg}(\lambda)=\lambda_1$ denote the {\it largest part} of partition $\lambda$, and set $\lambda_k=0$ if $k>\ell(\lambda)$.

\begin{theorem}\label{Cthm2} For $f\colon \mathbb N \to \mathbb C, q\in \mathbb C$ such that the product converges absolutely, we have
\begin{equation*}
\prod_{n=1}^{\infty}(1-f(n)q^n)^{-1}=\sum_{\lambda \in \mathcal S}q^{\operatorname{lg}(\lambda)}\prod_{i\geq 1}f(i)^{(\lambda_i-\lambda_{i+1})/i}.
\end{equation*}

\end{theorem}

%
%

\begin{proof}[Proof of Theorem \ref{Cthm2}]
For $j=1,2,3,...$, let $\mathcal P_{T_j}$ denote partitions whose parts are all in some subset $T_j\subseteq \mathbb N$, with $\emptyset \in \mathcal P_{T_j}$ for all $j$, and define $f_j: T_j \to \mathbb C$. To prove Theorem \ref{Cthm2}, we begin by recalling Corollary 2.9 of \cite{Schneider_zeta} in the case that ``$\pm$'' signs are set to minus: 
\[
\prod_{j=1}^{n}\prod_{k_j\in T_j}\left(1- f_j(k_j)q^{k_j}\right)^{- 1}
=\sum_{k=0}^{\infty}c_k q^k,
\]
with the coefficients $c_k$ given by the somewhat unwieldy $(n-1)$-tuple sum
\begin{multline*}
c_k= \sum_{k_2=0}^{k}\sum_{k_3=0}^{k_2}\dots\sum_{k_{n}=0}^{k_{n-1}}\left(\sum_{\substack{\lambda\vdash k_{n}\\ \lambda\in\mathcal P_{T_{n}}}}\prod_{\lambda_i\in\lambda}f_{n}(\lambda_i)\right)\left(\sum_{\substack{\lambda\vdash (k_{n-1}-k_{n})\\ \lambda\in\mathcal P_{T_{n-1}}}}\prod_{\lambda_i\in\lambda}f_{n-1}(\lambda_i)\right)\\
\times\left(\sum_{\substack{\lambda\vdash (k_{n-2}-k_{n-1})\\ \lambda\in\mathcal P_{T_{n-2}}}}\prod_{\lambda_i\in\lambda}f_{n-2}(\lambda_i)\right)\dots\left(\sum_{\substack{\lambda\vdash (k-k_2)\\ \lambda\in\mathcal P_{T_{1}}}}\prod_{\lambda_i\in\lambda}f_1(\lambda_i)\right),
\end{multline*}
where ``$\lambda \vdash r$'' indicates $\lambda$ is a partition of $r$ and the interior products are taken over the parts $\lambda_i$ of each $\lambda$, which identity can be proved from (\ref{Czeta_thm}) by repeated application of the Cauchy product formula. 

Now, for every $j\in \mathbb N$ take $T_j=\{j\}$ and fix $f_j= f$. In this case, $\lambda \in \mathcal P_{T_j}$ means if $\lambda \neq \emptyset$ that $\lambda=(j,j,...,j)$, so we must have $j|(k_j - k_{j+1})$ in any nonempty partition sum on the right side above. Then every summand comprising $c_k$ 
vanishes unless all the $k_i\leq k$ are parts of a sequentially congruent partition having length $\leq n$: 
each sum over partitions is empty (i.e., equal to zero) if $j$ does not divide $k_j-k_{j+1}$; is equal to 1 if $k_j - k_{j+1}=0$ as then $\lambda=\emptyset$ and $\prod_{\lambda_i\in \emptyset}$ is an empty product; or else has one term $f(j)^{m_j}=f(j)^{(k_j-k_{j+1})/j}$ as there is exactly one $\lambda = (j,j,...,j)$ with $|\lambda|=m_j j=k_j-k_{j+1}> 0$. Finally, let $n \to \infty$ so this argument encompasses partitions in $\mathcal S$ of unrestricted length.
\end{proof}

\begin{remark}
We note that setting $f=1$, then comparing equation (\ref{Czeta_thm}) to Theorem \ref{Cthm2}, gives another proof of Theorem \ref{Cthm1}: the sets $\mathcal S_{\operatorname{lg}=n}$ and $\mathcal P_n$ (and thus, the sets $\mathcal S$ and $\mathcal P$) have the same product generating function. \end{remark}

\begin{remark} If we instead take every $\pm$ equal to plus in Corollary 2.9 of \cite{Schneider_zeta}, similar arguments reveal there is also a bijection between partitions into {\it distinct} parts and the subset of $\mathcal S$ containing partitions into parts with differences $\lambda_i - \lambda_{i+1}=i$ exactly.
\end{remark}

\section{Cyclic Sequentially Congruent Maps}\label{Sect3}
Comparing Theorem \ref{Cthm2} with (\ref{Czeta_thm}) above, we have two formally different-looking decompositions of the coefficients of $\prod_{n\geq1}(1-f(n)q^n)^{-1}$ as sums over partitions of the form 
$\sum_{\lambda \in \mathcal P_n}$ and $\sum_{\lambda \in \mathcal S_{\text{lg}=n}}$, yet one observes the summands in each case consist of 
the same terms 
in different orders. 
Then one wonders: precisely {\it which} partition $\gamma \in \mathcal P_n$ is such that \begin{equation}\label{prodeq} \prod_{i \geq 1}f(i)^{(\phi_i-\phi_{i+1})/i}=\prod_{j \geq 1}f(j)^{m_j(\gamma)}\end{equation} for a given $\phi \in \mathcal S_{\text{lg}=n}$?
One observes that $\gamma$ is generally {not} the same partition $\lambda=\pi^{-1}(\phi)$ as in \eqref{piinv}. 

Evidently the set $\mathcal S$ enjoys a second map to $\mathcal P$ (apart from $\pi^{-1}$).   
Let $$\sigma\colon \mathcal S \to \mathcal P$$ denote this map. 
We can write $\sigma$ down by comparing the forms of the products in \eqref{prodeq}: 
$$\sigma(\phi):=(1^{\phi_1-\phi_{2}}\  2^{(\phi_2-\phi_{3})/2}\  3^{(\phi_3-\phi_{4})/3}...)=\gamma\in \mathcal P_n,$$
where $\phi \in \mathcal S_{\text{lg}=n}$ as above. For example, $\sigma(5,3,3)=(1^{5-3}\  2^{(3-3)/2}\  3^{(3-0)/3})=(3,1,1)$.

Under this map we have $\sigma(\mathcal S_{\text{lg}=n})=\mathcal P_n$, thus the composite map is 
$$\sigma \circ \pi\colon \mathcal P_n \to \mathcal P_n,$$ 
and, similarly, we have the map $\pi  \circ \sigma\colon \mathcal S_{\text{lg}=n} \to \mathcal S_{\text{lg}=n}.$ 

A natural question to ask is: what kind of permutation structure arises as we alternately compose $\pi, \sigma$, that is, 
what if we apply $\sigma \circ \pi \circ \sigma \circ \pi  \circ \cdots \circ  \sigma \circ \pi$ to a partition of $n$? 
For a concrete example, let's check by repeatedly applying $\sigma \circ \pi   \circ \cdots \circ  \sigma \circ \pi$ to the partitions of $n=4$: 
$$(4)\overset{\pi}\longmapsto (4)\overset{\sigma}\longmapsto (1,1,1,1)\overset{\pi}\longmapsto (4,4,4,4)\overset{\sigma}\longmapsto (4),$$
$$(3,1)\overset{\pi}\longmapsto (4,2)\overset{\sigma}\longmapsto (2,1,1)\overset{\pi}\longmapsto (4,3,3)\overset{\sigma}\longmapsto (3,1),$$
$$(2,2)\overset{\pi}\longmapsto (4,4)\overset{\sigma}\longmapsto (2,2),$$
$$(2,1,1)\overset{\pi}\longmapsto (4,3,3)\overset{\sigma}\longmapsto (3,1)\overset{\pi}\longmapsto (4,2)\overset{\sigma}\longmapsto (2,1,1),$$
$$(1,1,1,1)\overset{\pi}\longmapsto (4,4,4,4)\overset{\sigma}\longmapsto (4)\overset{\pi}\longmapsto (4)\overset{\sigma}\longmapsto (1,1,1,1).$$
There appears to be cyclic behavior of order 1 or 2; also evident is the following fact.

\begin{theorem}\label{Cthm3}
The composite map $\sigma \circ \pi\colon \mathcal P_n \to \mathcal P_n$ takes partitions to their conjugates. 
\end{theorem}

\begin{proof} 
If we write 
\begin{equation*}\label{Ceq1}
\lambda=(a_1^{m_{a_1}}a_2^{m_{a_2}}a_3^{m_{a_3}}...\  a_r^{m_{a_r}}),\  a_1>a_2>...>a_r\geq 1,
\end{equation*}
then we can compute the parts and frequencies of the conjugate partition $$\lambda^*=(b_1^{m_{b_1}}b_2^{m_{b_2}}b_3^{m_{b_3}}...\  b_s^{m_{b_s}}),\  b_1>b_2>...>b_s\geq 1,$$ directly from the parts and frequencies of $\lambda$ 
by comparing the Ferrers-Young diagrams of $\lambda,\lambda^*$. 
The conjugate partition $\lambda^*$ has  largest part $b_1$ given by
\begin{equation}\label{Cconjugate1}
b_1=\ell(\lambda)=m_{a_1}+m_{a_2}+...+m_{a_r},\  \  \text{with}\  \  m_{b_1}(\lambda^*)=a_r,
\end{equation}
and for $1<i\leq s$, the parts and their frequencies are given by
\begin{equation}\label{Cconjugate2}
b_i=m_{a_1}+m_{a_2}+...+m_{a_{r-i+1}},\  \  \  \  m_{b_i}(\lambda^*)=a_{r-i+1}-a_{r-i+2}.
\end{equation}
Moreover, we have that $s=r$. 
%
The theorem results from using the definitions of the maps $\pi$ and $\sigma$, keeping track of the parts in the transformation $\lambda \mapsto (\sigma \circ  \pi)(\lambda)$, then comparing the parts of $(\sigma \circ  \pi )(\lambda)$ with the parts of $\lambda^*$ in \eqref{Cconjugate1} and \eqref{Cconjugate2} above to see they are the same.
\end{proof}

The preceding considerations also make explicit our observation above about cyclic orders. 
\begin{corollary}
We have that $(\sigma \circ  \pi)(\lambda)=\lambda$ when $\lambda$ is self-conjugate, and $(\sigma  \circ  \pi)^2(\lambda)=\lambda$ holds for all $\lambda\in\mathcal P$. Likewise, for $\phi$ sequentially congruent it is the case that $(\pi  \circ \sigma)(\phi)=\phi$ when $\sigma(\phi)$ is self-conjugate, and $(\pi  \circ \sigma)^2(\phi)=\phi$ holds for all $\phi \in \mathcal S$. 
\end{corollary}

\begin{remark}
Interestingly, the map $\pi  \circ \sigma:\mathcal S_{\operatorname{lg}=n}\to \mathcal S_{\operatorname{lg}=n}$ defines a duality analogous to conjugation in $\mathcal P_n$, that instead connects partitions $\phi $ and $(\pi \circ \sigma)(\phi)$ in $\mathcal S_{\operatorname{lg}=n}$. For instance, from the above examples, it is the case in $\mathcal P_4$ that $(2,1,1)$ and $(3,1)=(\sigma \circ  \pi) (2,1,1)$ are conjugates, while on the same row, $(4,3,3)$ and $(4,2)=(\pi \circ  \sigma) (4,3,3)$ are paired under this new, analogous duality in $\mathcal S_{\operatorname{lg}=4}$.\end{remark}



\section{Frequency Congruent Partitions and Infinite Families of Bijections}\label{Sect4}
The conjugates of sequentially congruent partitions are themselves interesting combinatorial objects.
\begin{theorem}\label{multcongrthm}
A sequentially congruent partition $\phi$ 
is mapped by conjugation to a partition $\phi^*$ whose frequencies $m_i=m_i(\phi^*)$ obey the congruence condition \begin{equation*}
m_i\equiv 0\  (\operatorname{mod}\  i). \end{equation*}
Conversely, any partition with parts obeying this congruence condition has a sequentially congruent partition as its conjugate. 
\end{theorem}

\begin{proof} The theorem is immediate by conjugation of the relevant Young diagrams.
\end{proof}
Let us codify the objects highlighted in the preceding theorem.

\begin{definition} 
We define a partition to be {\it frequency congruent} if it has the property that each part divides its frequency\footnote{As in Theorem \ref{multcongrthm}}. \end{definition}

Then Theorem \ref{multcongrthm} implies the following result.

\begin{corollary}\label{multcor}
Frequency congruent partitions of length $n$ are in bijection with the partitions of $n$, viz. $$\#\{\lambda \in \mathcal P : \ell(\lambda)=n, \  i | m_i(\lambda)\}\  =\  p(n).$$
\end{corollary}
\vfill
\begin{proof} This statement follows from Theorem \ref{multcongrthm} together with Theorem \ref{Cthm1}. For a combinatorial proof, take any partition $\lambda=(1^{m_1} 2^{m_2}3^{m_3}...i^{m_i}...)$ of $n$, and multiply each $m_i$ by $i$ to yield a frequency congruent partition $(1^{ m_1}2^{2 m_2} 3^{3m_3}...i^{i m_i}...)$ with length $m_1+2 m_2 + 3 m_3 +...=|\lambda| = n$. Conversely, by the same principle, divide the frequency of each part of a length-$n$ frequency congruent partition by the part itself for a partition of $n$. 

Alternatively, we can prove the bijection using generating functions. 
For $|x|<1, |q|<1$, consider the following 
identities in light of \eqref{Czeta_thm} and \eqref{Czeta_thm3}: 
\begin{flalign*}
    &\prod_{n=1}^{\infty} \frac{1}{1 - x^{n}q^{n^{2}}} \\
    &= \   (1 + x^{1}q^{1} + x^{2}q^{1 + 1} +x^3 q^{1+1+1}+ ...) \\
    & \times (1 + x^2 q^{2 + 2} + x^4 q^{2 + 2 + 2 + 2} +x^6 q^{2+2+2 + 2 + 2 + 2}+ ...)  \\
    & \times (1 + x^3 q^{3 + 3 + 3}+ x^6 q^{3+3+3+3+3+3}+x^9 q^{3+3+3+3+3+3+3+3+3}+  ...) \times \cdots \\
    &=\sum_{\substack{\lambda\in\mathcal P \\   i | m_i(\lambda)}}x^{\ell(\lambda)}q^{|\lambda|}=\sum_{n=0}^{\infty}x^n \sum_{\substack{\ell(\lambda)=n \\   i | m_i(\lambda)}}q^{|\lambda|},
\end{flalign*}
where the final two (absolutely convergent) sums are taken over frequency congruent partitions.


To count the number of frequency congruent partitions of length $n$, let $q \to 1$ from within the unit the circle in the right-most series above, noting in the limit we still have convergence since $|x|<1$. Then by comparison with the product side of the generating function, the resulting coefficient of $x^{n}$ is equal to $p(n)$ by Euler's identity (see \cite{Andrews_theory}).
\end{proof}

 \begin{remark}
We note that the generating function proof above provides (by conjugation) another proof that $\#\mathcal S_{\text{lg}=n}=p(n)$. 
 \end{remark}

Indeed, the steps of the preceding proof suggest a highly general frequency congruence phenomenon yielding infinite families of partition bijections. 

As before, let $\mathcal P_T\subseteq \mathcal P$ be the set of partitions (including $\emptyset$) with parts from $T=\{t_1, t_2, t_3, ...\} \subseteq \mathbb N$; we allow $\mathcal P_T$ to also denote partitions with parts from a {sequence} $T$ of natural numbers if they are distinct. Let $p_T (n)$ denote the number of partitions of $n\geq 0$ in $\mathcal P_T$.
Moreover, for a sequence $S=(s_1, s_2, s_3,...)$ of natural numbers, define
\begin{flalign*}
\mathcal P_T(S):=\{\lambda\in \mathcal P_T : s_i | m_{t_i}\},  
\end{flalign*}
and let $\mathcal P_T(S,n)$ denote partitions in $\mathcal P_T(S)$ of length $n$. Thus $\mathcal P_{\mathbb N}((1,1,1,...))$ $=\mathcal P$ and $\#\mathcal P_{\mathbb N}((1,1,1,...),n)=p(n)$. Then we have the following.
\begin{theorem}\label{genthm}
Let $|x|< 1, |q|< 1$. For a sequence ${A}=(a_1, a_2, a_3,...)$ of natural numbers and subset $B=\{b_1, b_2, b_3, ...\}\subseteq \mathbb N$, we have 
%
$$ \prod_{n=1}^{\infty} \frac{1}{1 - x^{a_n}q^{a_n b_n}}\  =\  \sum_{\substack{\lambda\in\mathcal P_{B}(A)}}x^{\ell(\lambda)}q^{|\lambda|}\  =\  \sum_{n=0}^{\infty}x^n \sum_{\substack{\lambda\in\mathcal P_B(A,n)}}q^{|\lambda|}.$$
If the $a_i\in A$ are distinct then the sets $\mathcal P_A$ and $\mathcal P_B(A)$ are in bijection, and 
$$\#\mathcal P_B(A,n)\  =\  p_A(n).$$
\end{theorem}


We note that equation \eqref{Czeta_thm3} represents the case $a_i=1, b_i=i$, and the generating function in the proof of Corollary \ref{multcor} is the  case $a_i=b_i=i$.

\begin{proof}
For the first identity, much as in the proof of Corollary \ref{multcor}, for $|x|<1, |q|<1$, rewrite the infinite product on the left side of Theorem \ref{genthm} as a product of geometric series:
\begin{flalign*}
    \prod_{n=1}^{\infty}\left(1+x^{ a_n} q^{b_n+b_n+...+b_n} +x^{2 a_n} q^{b_n+...+b_n} +x^{3 a_n} q^{b_n+...+b_n} +...\right),
   \end{flalign*}
where in each term $x^{i a_n}q^{b_n+...+b_n}$ there are $ia_n$ repetitions of $b_n$ in the exponent of $q$. Expanding the product immediately gives the first equality, and collecting coefficients of $x^n$ gives the right-most equality.

To prove the second identity in the theorem, just as in the proof of Corollary \ref{multcor}, let $q\to 1$ from within the unit circle in the right-most summation of the first identity. 
But if the $a_i$ are distinct the infinite product becomes
$$ \prod_{n=1}^{\infty} \frac{1}{1 - x^{a_n}}=\prod_{n\in A} \frac{1}{1 - x^{n}}=\sum_{n=0}^{\infty}p_A(n)x^n.$$
Equating coefficients of $x^n$ completes the proof. 

One can also prove the second identity by mapping every partition $(a_1^{m_{a_1}}a_2^{m_{a_2}}a_3^{m_{a_3}}...)\in \mathcal P_A$ of size $n$ (noting these $a_i$ are not necessarily in increasing order) to partition $(b_1^{a_1m_{a_1}}b_2^{a_2m_{a_2}}b_3^{a_3m_{a_3}}...)\in \mathcal P_B(A,n)$ and, conversely, mapping each $(b_1^{a_1 n_1}b_2^{a_2n_{2}}b_3^{a_3n_{3}}...)\in \mathcal P_B(A,n)$ to $(a_1^{n_{1}}a_2^{n_{2}}a_3^{n_{3}}...)$.\footnote{There is a resemblance here to maps generating other classes of partitions with nontrivial weightings on the frequencies, e.g. see \cite{Alladi, Uncu, Dousse} regarding identities of Capparelli and Primc. } 
\end{proof}

Observe that in the above notation, frequency congruent partitions represent the set $\mathcal P_{\mathbb N}\left( (1,2,3,4,...) \right)$. Recalling that the conjugates of frequency congruent partitions are sequentially congruent, then the set $\mathcal S_B(A)$ of {\it conjugates of partitions in $\mathcal P_B(A)$} is evidently an analog of the set $\mathcal S$. For example, for $B=\mathbb N$ and sequence $A$, the conjugates of the set of partitions $\mathcal P_{\mathbb N}(A)$ such that $a_i$ divides $m_i$ have a nice sequential congruence property:
\begin{equation}\mathcal S_{\mathbb N}(A)=\{\lambda \in \mathcal P : \lambda_i \equiv \lambda_{i+1}\  (\text{mod}\  a_i)\}.\end{equation}
We conjecture there are bijective maps in this extended regime analogous to those in Sections \ref{Sect2} and \ref{Sect3} above; however, they alternate between $\mathcal P_A$ and $\mathcal S_B(A)$ under composition 
instead of between $\mathcal P$ and $\mathcal S$.

\begin{remark}
For $T\subseteq \mathbb N$ and $f,g:T\to \mathbb C$, two-variable generating functions of the general form $\prod_{n\in T}(1-x^{f(n)}q^{g(n)})^{-1}$ used in this section are flexible analytic and combinatorial objects (see \cite{Andrews_theory, Fine}). We note if $0<x<e^{-1}, |q|<1, 1 \not\in T$, taking $f(n)=\operatorname{log}\  n $ and letting $q\to 1$ as we did above yields a class of ``partition zeta functions'' studied in \cite{ORS, Schneider_zeta}: 
$$\lim_{q\to 1}\prod_{n\in T}(1-x^{\operatorname{log} n}q^{g(n)})^{-1}=\prod_{n\in T}(1-n^{\operatorname{log} x})^{-1}=\sum_{\lambda \in \mathcal P_T}{N(\lambda)^{-s}},$$
where $s:= -\operatorname{log} x$, thus $s>1$ for convergence, and $N(\lambda):=\prod_{\lambda_i \in \lambda}\lambda_i$. (By the same token, one may rewrite the Riemann zeta function as $\zeta(s)=\zeta(-\log x)=\sum_{n=1}^{\infty}x^{\log n}$.) 
\end{remark}

\section{Further Thoughts: Partition Ideals}

In a series of papers in the 1970s (e.g. see \cite{Andrews_computers, Andrews_computers2}), G. E. Andrews developed a {\it theory of partition ideals} which uses ideas from lattice theory to unify and extend many classical results on generating functions and partition bijections, summarized in Chapter 8 of \cite{Andrews_theory}. 
\begin{definition}\label{idealdef}
A partition ideal is a subset $\mathcal C \subseteq \mathcal P$ with the property that if any parts are deleted from a partition in $\mathcal C$, the resulting partition is an element of $\mathcal C$ as well. 
\end{definition}

\begin{remark}
We note Andrews's definition is stated in terms of frequencies.
\end{remark}

For example, partitions into distinct parts form a partition ideal. 
Andrews identifies relations between partition ideals which break the set $\mathcal P$ into algebraic subclasses.
\begin{definition}\label{equivdef}
We say two partition ideals $\mathcal C, \mathcal C'$ are equivalent and write $\mathcal C \sim \mathcal C'$ if $\#\{\lambda \in \mathcal C : |\lambda|=n\}=\#\{\lambda \in \mathcal C' : |\lambda|=n\}$
for all $n\geq 1$.\end{definition}

Andrews carries out the study of equivalences where one subset $\mathcal C$ is a partition ideal of ``order one'' in great detail (see \cite{Andrews_theory} for specifics). 
These are ``nice'' subsets of $\mathcal P$ including many of interest classically, e.g., partitions into distinct parts form a partition ideal of order one. 
Sets $\mathcal P_A, \mathcal P_B$ 
as in Theorem \ref{genthm} are also partition ideals of order one. 
Naturally, then, one wonders if Andrews's theory extends in some way to sets like $\mathcal P_B(A)$. 

A moment's thought convinces one that such sets are {not} generally partition ideals. However, they do enjoy a tantalizing ``quasi-ideal'' property: 
{\it If $a_i$ copies (or a multiple thereof) of any part $b_i$ are deleted from a partition in $\mathcal P_B(A)$, the resulting partition is an element of $\mathcal P_B(A)$ as well.} %

This feels like a refinement of Definition \ref{idealdef}. Furthermore, if the  $a_i$ in the sequence $A$ of distinct terms are rearranged to form a new sequence $A'$ (the same terms in a different order), clearly $p_{A'}(n)=p_A(n)$ even though $\mathcal P_B(A')\neq \mathcal P_B(A)$; thus Theorem \ref{genthm} gives \begin{equation}\label{eqA} \#\mathcal P_B(A,n)=\#\mathcal P_B(A',n).\end{equation}
Similarly, noting $B$ is arbitrary in Theorem \ref{genthm} and could be replaced by another subset $B'\subseteq\mathbb N$ without changing the right side of the second identity, then \begin{equation}\label{eqB}\#\mathcal P_B(A,n)=\#\mathcal P_{B'}(A,n).\end{equation}
In light of the correspondence between length-$n$ partitions in $\mathcal P_B(A)$ and {size}-$n$ partitions in $\mathcal P_A$, equations \eqref{eqA} and \eqref{eqB} feel similar to partition ideal equivalence in Definition \ref{equivdef}. 

Moreover, the two-variable generating functions in Section \ref{Sect4} are of a similar shape to Andrews's formulas for ``linked partition ideals'' in Chapter 8.4 of \cite{Andrews_theory}. Are there maps between these schemes? If subsets of partitions such as $\mathcal P_B(A)$ are analogous to partition ideals, do there exist closely-related subsets analogous to 
equivalent ideals in Andrews's theory? Conversely, might cyclic maps like those in Section \ref{Sect3} exist between equivalent partition ideals? 

\subsection*{Acknowledgment} The authors are grateful to the organizers of the Combinatory Analysis 2018 conference and the editors of these proceedings, and to the anonymous referee for many useful comments and references. Furthermore, the second author would like to thank George E. Andrews and Andrew V. Sills for conversations that informed this work.

\end{document}